\documentclass[11pt,a4paper]{article}

\usepackage[all,cmtip]{xy}
\entrymodifiers={+!!<0pt,\fontdimen22\textfont2>}

\usepackage{fouriernc}
\usepackage[english]{babel}         
\usepackage{verbatim}               
\usepackage[T1]{fontenc}

\usepackage{amsmath}
\usepackage{amsfonts}
\usepackage{amssymb}
\usepackage{mathrsfs}    
\usepackage{amsthm}

\usepackage{todonotes}

\usepackage{hyperref}           

\swapnumbers

\theoremstyle{plain}
\newtheorem*{thrm}{Theorem}
\newtheorem{thm}{Theorem}[section]

\newtheorem{prop}[thm]{Proposition}
\newtheorem{cor}[thm]{Corollary}
\newtheorem{lemma}[thm]{Lemma}

\theoremstyle{definition}
\newtheorem*{defnn}{Definition}
\newtheorem{defn}[thm]{Definition}

\theoremstyle{remark}
\newtheorem{remark}[thm]{Aside}
\newtheorem{eg}[thm]{Example}
\newtheorem{egs}[thm]{Examples}
\newcommand{\proofof}[1]{\end{#1}\begin{proof}}

\makeatletter
\renewcommand\section{\@startsection {section}{1}{\z@}%
  {-3.5ex \@plus -1ex \@minus -.2ex}{2.3ex \@plus.2ex}%
  {\normalfont\large\bfseries}}
\renewcommand\subsection{\@startsection{subsection}{2}{\z@}%
  {-3.25ex\@plus -1ex \@minus -.2ex}{1.5ex \@plus .2ex}%
  {\normalfont\bfseries}}

\newcommand{\sh}[1]{\mathcal{#1}}
\newcommand{\N}{{\mathbb N}}
\newcommand{\Z}{{\mathbb Z}}
\newcommand{\Q}{{\mathbb Q}}
\newcommand{\R}{{\mathbb R}}

\newcommand{\F}{{\mathbb F}}
\newcommand{\B}{{\mathbb B}}

\newcommand{\tens}{\mathbin{\otimes}}

\DeclareMathOperator*\colim{colim}
\DeclareMathAlphabet{\mathrmsl}{OT1}{cmr}{m}{sl}

\newcommand{\rssymb}[2]{\newcommand{#1}{\mathrmsl{#2}} }
\newcommand{\oper}[3][n]{\newcommand{#2}{\mathop{\mathrm{#3}}%
\ifx n#1\nolimits\else\limits\fi} }
\newcommand{\rsoper}[3][n]{\newcommand{#2}{\mathop{\mathrmsl{#3}}%
\ifx n#1\nolimits\else\limits\fi} }

\oper\Ad{Ad}
\oper\ad{ad}

\oper\val{val}
\oper\coker{coker}
\oper\mult{mult}
\oper\Iso{Iso}
\oper\End{End}
\oper\Aut{Aut}
\oper\Sub{Sub}
\oper\Alt{Alt}
\oper\Ext{Ext}
\oper\Pic {Pic}
\oper\Sym{Sym}
\oper\Spec{Spec}
\oper\Spf{Spf}
\oper\Sp{Sp}
\oper\Spa{Spa}
\oper\Proj{Proj}
\rsoper\divg{div}
\rsoper{\sym}{sym}
\rsoper{\alt}{alt}
\rsoper\trace{tr}
\rssymb\id{id}

\newcommand{\thismonth}{\ifcase\month\or
  January\or February\or March\or April\or May\or June\or
  July\or August\or September\or October\or November\or December\fi
  \space\number\year}

\newcommand{\CPA}{\mathrm{CPA}}
\newcommand{\Mod}{\mathrm{Mod}}

\title{On the difference between `tropical functions' and real-valued functions}
\author{Andrew W. Macpherson}
\begin{document}

\maketitle
\begin{abstract}
I introduce the concept of integral closure for elements and ideals in idempotent semirings, and establish how it corresponds to its namesake in commutative algebra. 

In the case of free semirings, integrally closed elements give canonical `maximal' representatives for the associated piecewise-affine functions on Euclidean space.
\end{abstract}

Idempotent semirings arise in nature as the co-ordinate algebras of tropical varieties, skeletons of analytic spaces, and singular affine manifolds. Recent years have seen more and more attempts to understand the structure of these objects - especially of tropical varieties - via idempotent semialgebra; whether extrinsically, through tropical Nullstellens\"atze \cite{IzhakianShustin, BertramEaston}, or intrinsically, though co-ordinate semirings \cite{Giansiracusa,IzhakianRowen,MaclaganRincon}.

Our geometric intuition for idempotent semirings comes from sets of convex, piecewise-affine functions with integral slopes, which form semirings under the operations of $\vee$ (pointwise max) and $+$ plus. 
Take, for example, the affine manifold $\R^n$. Its function semiring $\mathrm{CPA}_\Z(\R^n,\R)$ is generated under $\vee$ by affine functions, which in turn are generated, as an additive group, by the co-ordinates $X_i:\R^n\rightarrow\R$ and real constants.

When trying to understand the algebraic structure of $\mathrm{CPA}(\R^n,\R)$, our first guess might be that it is equal to the \emph{free semiring} $\R_\vee[\pm X_i]_{i=1}^n$ whose elements are `tropical polynomials': formal finite combinations $\bigvee_{k\in\Z^n}k\cdot X+\lambda_k$ with $\lambda_k\in\R$. However, when we investigate the evaluation homomorphism
\[ \R_\vee[\pm X_i]_{i=1}^n \rightarrow \mathrm{CPA}(\R^n,\R), \]
we find something odd: it is not injective! For example, 
\begin{align}\label{intro1} 2(0\vee X)=&0\vee X \vee 2X = 0\vee 2X \end{align}
as convex functions, while the second identity fails to hold in the free semiring. 

Similarly, while the function semiring evidently satisfies the cancellative law for addition, the source does not: while
\begin{align} \label{intro2} 3(0\vee X) &= (0\vee 2X) + (0\vee X) \end{align}
in any semiring, cancelling $(0\vee X)$ from both sides would reduce to the identity \eqref{intro1} which we have just observed is not true in $\R_\vee[\pm X_i]_{i=1}^n$.

These issues with the tropical polynomial ring have been known to the community for a long time. There have been some recent attempts to address them algebraically \cite[\S3.3]{BertramEaston}, \cite[\S2.2]{Izhakian}. 

The approach taken in this paper is more general: we will define a notion of \emph{integral closure} for elements in any idempotent semiring, and say that a semiring is \emph{normal} if every element is integrally closed. In the special case of free semirings, the integrally closed elements are `maximal' representatives of honest real-valued functions.

\begin{thrm}[\ref{MON_THM}]Let $\Delta$ be a polytope inside an affine space $N$,  $Q$ the group of affine functions on $N$, and $Q^+\subseteq Q$ the submonoid of functions less than or equal to zero on $\Delta$. If $\B[Q;Q^+]$ denotes the free semiring of the pair $(Q;Q^+)$ (defined \S\ref{INT_MON}), then
\[  \B[Q;Q^+] \rightarrow \mathrm{CPA}_\Z(\Delta,\R) \]
is a normalisation of idempotent semirings.\footnote{In the context of this paper, we prefer to work with bounded polyhedral subsets of tropical space rather than the space itself; $\B[Q;Q^+]$ is an accordingly modified version of the free semiring.}\end{thrm}

In light of this result, we see that in the case of a free semiring our notion of integrally closed element coincides with the \emph{saturated} tropical polynomials of \cite[\S3.3.2]{BertramEaston}.

\

The definition of integral closure in semirings is inspired by its namesake in commutative algebra, for which my main reference is \cite{SwansonHuneke}.

\begin{defnn}[{\cite[Def. 1.1.1]{SwansonHuneke}}]An ideal $I\trianglelefteq A$ is \emph{integrally closed} if all elements $z\in A$ \emph{integral over I}, that is, satisfying a monic polynomial
\[ 0=z^n+c_1z^{n-1}+\cdots + c_n,\quad c_i\in I^i  \]
are already contained in $I$.\end{defnn}

Assume, for sake of exposition, that $A$ is a domain. There is a simple geometric intuition behind the above definition: a function is integral over $I$ if and only if it lands in $I$ after pullback to some birational modification of $\Spec A$. In other words, two integrally closed ideals that become equal on some blow-up of $\Spec A$ are already equal in $A$. By blowing up to make any finite set of ideals into Cartier divisors, we can thereby show that the product operation on integrally closed ideals is \emph{cancellative}.

We may couch this observation in terms of semiring theory as follows: we write $\B[K;A]$ for the set of fractional ideals of the quotient field $K$ of $A$. It has the structure of an idempotent semiring with max and plus defined using ideal sum and product, respectively.

The set $\B^\nu[K;A]$ of integral closures of fractional ideals is naturally a quotient semiring of $\B[K;A]$, with the quotient map being the integral closure operator. Our previous observation amounts to the fact that plus (ideal product) is cancellative on $\B^\nu[K;A]$. Thus, integral closure is in this case again related to cancellativity of a semiring.

More precisely:

\begin{thrm}[\ref{citeme}]Let $A$ be a domain with fraction field $K$. The integral closure operator
\[ \B[A;A^+] \rightarrow \B^\nu[A;A^+] \]
is a normalisation of idempotent semirings.\end{thrm}

\section*{Notation and conventions}

My notations mainly follow those of \cite{mac}; to save the reader leafing though that paper, I recap much of the basic notations here, along with some innovations.

\begin{itemize}
\item Idempotent semirings are commutative algebras in the category $\Mod_\B$ of idempotent monoids, or of join semilattices, if you like. My notation $(\vee,-\infty,+,0)$ for semiring operations follows the intuition of the \emph{max-plus algebra}, as in the introduction. Accordingly, the tensor operation on $\Mod_\B$ is denoted $\oplus$ rather than $\otimes$ - beware that it is \emph{not} a direct sum operation on $\Mod_\B$.

All semirings in this paper are idempotent. The basic examples are the initial \emph{Boolean semifield} $\B=\{-\infty,0\}$ and the \emph{rank one (max-plus) semifields} $\Z_\vee,\Q_\vee,\R_\vee$, where the subscript $\vee$ means `tack on $-\infty$', the identity for the operation $\vee$.

\item Semirings and their modules are in particular partially ordered sets, and so terminology from order theory carries over here. An \emph{(order) ideal} is a lower set closed under $\vee$. There is no requirement that it be stable under the action of any semiring, and thus should not be confused with the notion of \emph{semiring ideal}.

The ideal generated by two ideals $\iota_1,\iota_2$ is called their \emph{join}, and denoted $\iota_1\vee\iota_2$. If $\iota_i$ are principal generated by elements $X_i$, then $\iota_1\vee\iota_2$ is generated by $X_1\vee X_2$.

\item The set of all order ideals of $\mu$ is denoted  $\sh L(\mu)$; it is a complete lattice (whence the letter $\sh L$). A homomorphism $\mu_1\rightarrow\mu_2$ induces adjoint image and preimage monotone maps $\sh L(\mu_1)\leftrightarrows\sh L(\mu_2)$. (The image of an order ideal may fail to be lower; one must take the lower set generated by the image to define the covariance of $\sh L$.)

\item The set of elements less than zero in an idempotent semiring $\alpha$ is called its \emph{semiring of integers}, and denoted $\alpha^\circ$.

\item Localisations of semirings are defined as usual for commutative algebras in a monoidal category; note only that as a consequence of our notation, a localisation of an $\alpha$-module $\mu$ at a set of elements $S\in\alpha$ adjoins \emph{additive} inverses to $S$:\[\mu\rightarrow\mu[-S].\]
Semirings are assumed to satisfy the \emph{Tate condition}, which is that $\alpha$ is a localisation of $\alpha^\circ$. This is equivalent to requiring every element of $\alpha$ to be bounded above by some invertible element. An element that is also bounded below by an invertible element is said to be \emph{bounded}. In particular, all invertible elements are bounded.

When $S\in \alpha$ is bounded, the localisation $\alpha[-S]$ also satisfies the Tate condition. We  call this a \emph{bounded localisation}. For more information about bounded localisation, cf.\ \cite[\S5]{mac}. (In \emph{loc.\ cit.}\ another kind of localisation is discussed, called \emph{cellular}, which we do not use in this paper.)

\item The localisation of a semiring $\alpha$ at the system of all bounded elements is called the \emph{bounded difference semiring}, denoted $\psi:\alpha\rightarrow \mathrm{PL}(\alpha)$. 

\end{itemize}

\section*{Acknowledgements}
This paper was produced  under the funding of a Hodge fellowship, provided by W.\ D.\ Hodge Foundation. I would like to thank IH\'ES also for providing a peaceful environment for the completion of this work.
I'd also like to thank an anonymous referee for motivating me to make this a better paper.

\section{Integral closure}\label{INT}
Throughout this section, we fix an idempotent semiring $\alpha$ and an $\alpha$-module $\mu$.

\begin{prop}\label{EXT_PROPS}Let $\iota_1\subseteq\iota_2\subseteq\mu$ an inclusion of ideals of $\mu$. The following are equivalent:
\begin{enumerate}\item every element of $\iota_2$ maps into the image of $\iota_1$ in some bounded localisation of $\mu$;
\item the images of $\iota_1$ and $\iota_2$ in $\mu\oplus_\alpha\mathrm{PL}(\alpha)$ are equal;
\item for each $X\in\iota_2$, there exists bounded $S\in\alpha$ such that
\[ X+S\in\iota_1+S. \]
\end{enumerate}
When $\iota_2$ is finitely generated, we may add:
\begin{enumerate}
\item[iv)] $\iota_1$ and $\iota_2$ become equal in some bounded localisation of $\mu$;
\item[v)] there exists bounded $S\in\alpha$ such that
\[ \iota_1+S=\iota_2+S. \]\hfill\end{enumerate}\end{prop}
\begin{proof}Since $\mathrm{PL}(\alpha)$ is, by definition, the filtered colimit of all bounded localisations of $\alpha$, \emph{i)} and \emph{ii)} are equivalent by the construction of filtered colimits of sets.

Condition \emph{iv)} implies \emph{i)}. Conversely, if a set of elements of $\iota_2$ lands in $\iota_1$ after some bounded localisation, then the same is true for the ideal that they generate. Thus, if $\iota_2$ is finitely generated, \emph{i)} implies \emph{iv)}.

Condition \emph{iii)}, resp.\ \emph{v)}, falls out of condition \emph{i)}, resp.\ \emph{iv)}, by writing down explicitly what it means for two elements to become equal in the localisation $\mu[-S]$.\end{proof}

\begin{defn}An inclusion $\iota_1\subseteq \iota_2$ of ideals of $\mu$ is said to be an \emph{integral extension} if the equivalent conditions of proposition \ref{EXT_PROPS} are satisfied. We also say that $\iota_2$ is \emph{integral over} $\iota_1$.

Similarly, we say that an element $X_2$ is integral over $X_1$ if this is the case for the ideals they generate.\end{defn}

The following facts are immediate consequences of the definition, and are easiest to see starting from criterion \emph{ii)} of proposition \ref{EXT_PROPS}:
\begin{itemize}
\item Because taking the image commutes with colimits, any join of integral extensions is an integral extension.
\item The notion of integral extension is transitive: if $\iota_1\subseteq\iota_2$ and $\iota_2\subseteq\iota_3$ are integral extensions, then $\iota_1\subseteq\iota_3$ is an integral extension.
\end{itemize}
From the first item, it follows that there is a \emph{largest} integral extension of any ideal $\iota$; in light of criterion \emph{ii)} of proposition \ref{EXT_PROPS}, this is simply the preimage of its image in $\mu\oplus_\alpha\mathrm{PL}(\alpha)$. By the second item, this extension has no non-trivial integral extensions of its own.

In other words, our notion of integral extension of ideals is accompanied by a reasonable operation of \emph{integral closure}.

\begin{prop}\label{IDEAL_PROPS}Let $\mu$ be an $\alpha$-module, $\iota\subseteq\mu$ an ideal. The following are equivalent:
\begin{enumerate}\item every ideal integral over $\iota$ is contained in $\iota$;
\item $\phi^{-1}\phi\iota=\iota$ for all bounded localisations $\phi:\mu\rightarrow\mu[-S]$;
\item $\psi^{-1}\psi\iota=\iota$ where $\psi:\mu\rightarrow\mu\oplus_\alpha\mathrm{PL}(\alpha)$;
\item for every ideal $\iota^\prime\subseteq\mu$ and $S\in\alpha$,
\begin{align}\label{equations} \iota^\prime+S\leq \iota+S \quad&\Rightarrow\quad\iota^\prime\leq \iota \end{align}\end{enumerate}\end{prop}
\begin{proof}Conditions \emph{ii-iv)} are obtained by interpreting the first according to definitions \emph{i-iii)}, in that order, of proposition \ref{EXT_PROPS}.\end{proof}

\begin{defn}An ideal $\iota\subseteq\mu$ is said to be \emph{integrally closed} when the equivalent conditions of proposition \ref{IDEAL_PROPS} are satisfied.

An \emph{integral closure} of $\iota$ is an integrally closed ideal, integral over $\iota$. We have seen that every ideal has a unique integral closure.\end{defn}

Note two immediate consequences of the definition:
\begin{itemize}\item If $\iota\subseteq\mu$ is integrally closed, then so is its image in any bounded localisation of $\mu$.
\item Since both image and inverse image commute with filtered colimits, a filtered union of integrally closed ideals is integrally closed.\end{itemize}

Denote by $\sh L^\nu(\mu)\subseteq\sh L(\mu)$ the set of integrally closed ideals of $\mu$. Because every ideal has an integral closure, this subset is \emph{reflective}. It is also stable for filtered suprema.

Let $f:\mu_1\rightarrow\mu_2$ be an $\alpha$-module homomorphism. There are induced adjoint image and preimage maps between the $\sh L^\nu(\mu_i)$, which are compatible with the maps on $\sh L$ in the following sense:

\begin{lemma}[Functoriality of $\sh L^\nu$]
\

\begin{itemize}\item Let $\iota\subseteq\mu_1$ be an ideal. The integral closure of the image of $\iota$ in $\mu_2$ is integral over the image in $\mu_2$ of its integral closure in $\mu_1$. The corresponding square
\[\xymatrix{ \sh L(\mu_1)\ar[r]\ar[d] & \sh L(\mu_2)\ar[d] \\ \sh L^\nu(\mu_1)\ar[r] & \sh L^\nu(\mu_2)}\]
commutes (cf.\ \cite[Rmk.\ 1.1.3, (7)]{SwansonHuneke}). \hfill\emph{persistence property}
\item  If $\iota\subseteq\mu_2$ is integrally closed, then so is $f^{-1}\iota\subseteq\mu_1$. The square
\[\xymatrix{ \sh L^\nu(\mu_2)\ar[r]^{f^{-1}}\ar[d] & \sh L^\nu(\mu_1)\ar[d] \\ \sh L(\mu_2)\ar[r]^{f^{-1}} & \sh L(\mu_1)}\]
commutes (cf.\ \cite[Rmk.\ 1.1.3, (8)]{SwansonHuneke}). \hfill \emph{contraction property}
\item Let $\iota_i\subseteq\mu_i$. The integral closure of $\iota_1+\iota_2$ in $\mu_1\oplus_\alpha\mu_2$ is integral over $\iota_1^\nu+\iota_2^\nu$. The square
\[\xymatrix{\sh L(\mu_1)\times\sh L(\mu_2)\ar[r]\ar[d] &\sh L(\mu_1\oplus_\alpha\mu_2) \ar[d] \\
\sh L^\nu(\mu_1)\times\sh L^\nu(\mu_2)\ar[r] & \sh L^\nu(\mu_1\oplus_\alpha\mu_2)}\]
commutes (cf.\ \cite[Rmk.\ 1.3.2, (4)]{SwansonHuneke}).\hfill \emph{sum property}
\end{itemize}\end{lemma}
\begin{proof}(persistence property) Indeed, if $\iota\subseteq\iota^\prime$ is an integral extension in $\mu_1$, then so is its image in $\mu_2$, by the naturality of localisations.

(contraction property) By inspecting the square
\[\xymatrix{ \mu_1 \ar[r]\ar[d] & \mu_2\ar[d] \\ \mu_1\oplus_\alpha\mathrm{PL}(\alpha)\ar[r] & \mu_2\oplus_\alpha\mathrm{PL}(\alpha) }\] we see that $f^{-1}\iota$ is the pullback of an ideal from $\mu_1\oplus_\alpha\mathrm{PL}(\alpha)$.

(sum property) Let $\iota_i\subseteq\mu_i$, and let $\iota_i\subseteq\iota_i^\prime$ be a pair of integral extensions. By the compatibility of localisation with tensor product, $\iota_1+\iota_2\subseteq\iota_1^\nu+\iota_2^\nu$ is an integral extension inside $\mu_1\oplus\mu_2$.\end{proof}

As a special case of the contraction property:

\begin{cor}If $\mu_1\subseteq\mu_2$ is a submodule, an ideal $\iota\subseteq\mu_1$ is integrally closed in $\mu_1$ if and only if it is integrally closed in $\mu_2$ (cf.\ \cite[Rmk.\ 1.1.3, (9)]{SwansonHuneke}).\end{cor}

\section{Normal modules}\label{NORM}
\begin{prop}\label{INT_NORMALISATION}
Let $\mu$ be an $\alpha$-module. The following are equivalent:
\begin{enumerate}\item every element of $\mu$ is integrally closed;
\item every ideal of $\mu$ is integrally closed;
\item $\mu\rightarrow\mu\oplus_\alpha\mathrm{PL}(\alpha)$ is injective;
\item bounded elements of $\alpha$ act cancellatively on $\mu$;
\item $\mu\hookrightarrow\sh L(\mu)$ is bijective onto $\sh L^\nu(\mu)\subseteq\sh L(\mu)$.
\end{enumerate}\end{prop}
\begin{proof}Let $X,Y\in\mu$ and suppose that their images in $\mu\oplus_\alpha\mathrm{PL}(\alpha)$ are equal. Then the integral closure of $X$ equals the integral closure of $Y$. If \emph{i)} is satisfied, then this implies $X=Y$. Therefore, \emph{i)} implies \emph{iii)}.

Condition \emph{iii)} implies \emph{ii)}, by part \emph{iii)} of proposition \ref{IDEAL_PROPS}. 

Condition \emph{i)} is a special case of \emph{ii)}, and \emph{v)} is a rephrasing of \emph{i)}.

The equivalence of \emph{iv)} is obtained by applying the inequalities \eqref{equations} (part \emph{iv)} of proposition \ref{IDEAL_PROPS}) for integral closure of elements in both directions.\end{proof}

\begin{defn}An $\alpha$-module $\mu$ is said to be \emph{normal} if it satisfies the equivalent conditions of proposition \ref{INT_NORMALISATION}. A semiring is said to be normal if it is so as a module over itself.

The full subcategory of $\Mod_\alpha$ spanned by the normal modules is denoted $\Mod_\alpha^\nu$.\end{defn}

\begin{cor}\label{CANCELLATIVE}A cancellative semiring is normal.\end{cor}
\begin{proof}By criterion \emph{iv)} of the proposition.\end{proof}

\begin{cor}Any module over a semifield is normal.\end{cor}

Let $\mu$ be any $\alpha$-module. From condition \emph{iii)} of proposition \ref{INT_NORMALISATION}, one can see that the image of $\mu$ in $\mu\oplus_\alpha\mathrm{PL}(\alpha)$ is normal, and that it is initial among normal $\alpha$-modules under $\mu$. Thus it is a reasonable \emph{normalisation} $\mu^\nu$ of $\mu$. Note that the map $\mu\rightarrow \mu^\nu$ is \emph{surjective}.

Alternatively, by \emph{v)}, it is the image of the composite
\[ \mu \rightarrow \sh L(\mu) \rightarrow \sh L^\nu(\mu), \]
where the second arrow is the integral closure operator of \S\ref{INT}.

Normalisation is left adjoint to the inclusion of $\Mod_\alpha^\nu$ into $\Mod_\alpha$; in other words, the former is a reflective subcategory of the latter.

\paragraph{Limits and colimits} Since $\Mod_\alpha^\nu$ is a reflective subcategory of $\Mod_\alpha$,
\begin{itemize}\item any limit of normal modules is normal;
\item in particular, if $\alpha$ is normal, then so is any finite free $\alpha$-module;
\item normalisation commutes with colimits.\end{itemize}
Since filtered colimits are constructed in the category of sets, by item \emph{iv)} of proposition \ref{INT_NORMALISATION},
\begin{itemize}\item a filtered colimit of normal modules is normal.\end{itemize}
A coequaliser of normal modules needn't be normal; cf.\ example \ref{INT_EGS}.

\paragraph{Tensor sum}
The tensor sum of two normal modules need not be normal (cf.\ example \ref{INT_EGS}); to get a monoidal structure on $\Mod^\nu$, then, it must be normalised.

\begin{defn}The \emph{normalised tensor sum} of two modules is the normalisation of their tensor sum.\end{defn}

By the sum property of integral extensions, the normalised tensor sum of two modules equals the normalised tensor sum of their normalisations. Hence, the reflection functor $\Mod_\alpha\rightarrow\Mod_\alpha^\nu$ is strongly monoidal with respect to this structure. 

\paragraph{Change of semiring}
Let $\alpha\rightarrow\beta$ be a semiring homomorphism. 
\begin{itemize}\item The normalised tensor sum allows us to define a \emph{normalised base change}, making the diagram
\[\xymatrix{ \Mod_\alpha\ar[r]\ar[d] & \Mod_\beta\ar[d] \\ \Mod_\alpha^\nu\ar[r] & \Mod_\beta^\nu }\]
commute. \hfill\emph{scalar extension}

\item Conversely, if $\mu$ is a normal $\beta$-module, then $\mu$ is normal as an $\alpha$-module by part \emph{iv)} of proposition \ref{INT_NORMALISATION}. That is, the square
\[\xymatrix{ \Mod^\nu_\beta\ar[r]\ar[d] & \Mod^\nu_\alpha\ar[d] \\ \Mod_\beta \ar[r] & \Mod_\alpha }\]
commutes. \hfill\emph{scalar restriction}

\item Any bounded localisation of a normal module is normal (over either the original semiring or the corresponding localisation). \hfill\emph{localisation}\end{itemize}

Note also that, since $\alpha\rightarrow\alpha^\nu$ is surjective, $\alpha^\nu$ has a natural structure of a semiring, and the category of $\alpha^\nu$-modules is a full subcategory of $\Mod_\alpha$. Any normal module is an $\alpha^\nu$-module (but not conversely, cf.\ \ref{INT_EGS}).

\paragraph{Monics and epics}
Because $\Mod^\nu$ is a reflective subcategory of $\Mod$, a morphism is monic in $\Mod^\nu$ if and only if it is injective. Similarly, a morphism is epic if and only if it is surjective - this follows from the fact that the unit of the adjunction is surjective.

Moreover:
\begin{itemize}\item By the contraction property of integrally closed ideals, any submodule of a normal module is normal.\hfill \emph{hereditary}
\item Normalisation preserves monomorphisms. Indeed, since $\mathrm{PL}(\alpha)$ is a localisation of $\alpha$, it is flat and so $\mu_1\oplus_\alpha\mathrm{PL}(\alpha)\rightarrow\mu_2\oplus_\alpha\mathrm{PL}(\alpha)$ is injective. The claim then follows from the commutativity of \[\xymatrix{ \mu_1^\nu\ar[r]\ar[d] & \mu_2^\nu\ar[d] \\ \mu_1\oplus_\alpha\mathrm{PL}(\alpha)\ar[r] & \mu_2\oplus_\alpha\mathrm{PL}(\alpha) }\]
\hfill \emph{mono-flat}\end{itemize}

\

In summary:

\begin{prop}[Properties of the category of normal modules]The category of normal modules is a reflective subcategory of the category of all modules, stable under filtered colimits and scalar restriction. Any submodule of a normal module is normal.

The normalisation functor is strongly monoidal and commutes with base change. It preserves monomorphisms.\end{prop}

\begin{egs}\label{INT_EGS}Let $\alpha$ be the semiring of convex, piecewise-affine functions on an interval with co-ordinate $X$. Since $\alpha$ is cancellative, it is normal (corollary \ref{CANCELLATIVE}). 

Suppose that $X=0$ in the interior of the interval, and let $\mu$ be the $\alpha$-module with presentation
\[ \mu:=\alpha^2/((X\vee0,-\infty)=(-\infty,X\vee0)). \]
Neither basis vector of $\mu$ is integrally closed, and the integral closure of either is the whole of $\mu$. Thus the normalisation of $\mu$ is free of rank one. However, $\alpha^2$ is finite free and hence normal. Thus, a coequaliser of normal modules may be abnormal.

Continuing with this example, let $\alpha\rightrightarrows\mu$ be the inclusions of the basis elements. The equaliser of this pair is the semiring ideal generated by $X\vee 0$. However, after normalisation, both maps become equal. Thus normalisation is not left exact: it does not preserve equalisers.

Finally, although $\alpha$ is normal, the un-normalised tensor double $\alpha\oplus_{\R_\vee}\alpha$ is not; it suffers from the same anomaly \eqref{intro1}
\begin{align}\nonumber (2X\oplus 0\vee 0\oplus 2X) + (X\oplus 0\vee 0\oplus X) &= 3(X\oplus 0 \vee0\oplus X) \\
\nonumber\text{while}\qquad 2X\oplus 0 \vee0\oplus 2X & \neq 2(X\oplus0\vee0\oplus X) \end{align}
that we saw in the introduction.\end{egs}

\section{Normal semirings}\label{INT_SEMIRINGS}

Arithmetic in normal semirings enjoys a few simplifications:

\begin{lemma}\label{SEMIRINGS_LEMMA}Let $\alpha$ be a normal semiring, $X,Y\in\alpha$. Suppose $Y$ is bounded. Then:\begin{align}\label{the_implication} n(X\vee Y) \leq (n-1)(X\vee Y) + Y \quad &\Rightarrow \quad X\leq Y & \textit{(reduction)} \\
n(X\vee Y) \quad &=\quad nX\vee nY  & \textit{(binomial)} \\
\label{divisible} nX\leq nY \quad&\Rightarrow\quad X\leq Y &\textit{(divisibility)}
 \end{align}\end{lemma}

The reader will note that each of these statements holds for a semiring of convex, piecewise-affine functions, but fails for a free semiring (see the introduction). For reduction \eqref{the_implication}, compare the \emph{reduction criterion} of \cite[\S1.2]{SwansonHuneke}.

\begin{proof}The first implication is by cancellation of the bounded element $(n-1)(X\vee Y)$. For the binomial: 
\begin{align} \nonumber n(X\vee Y) &= 2n(X\vee Y) - n(X\vee Y) = (nX\vee nY) + n(X\vee Y) - n(X\vee Y)  \\\nonumber &= nX\vee nY.\end{align}
Finally, by combining reduction with the binomial identity, $nX\leq nY$ implies that $n(X\vee Y)=nY\leq (n-1)(X\vee Y)+Y$.\end{proof}

The reduction and divisibility implications may alternatively be interpreted as follows: if $\alpha$ is any semiring - not necessarily normal - then the inequality on the left implies the inequality on the right for the \emph{integral closures} of $X$ and $Y$. If, in particular, $Y\leq X$, then the inequality on the left implies that $X$ is an integral extension of $Y$.

The reduction inequality may be rewritten
\[ nX \leq \bigvee_{i=0}^{n-1}iX + (n-i)Y, \]
which may be expressed by saying that $X$ obeys an `equation of integral dependence' over $Y$; compare the definition of integral dependence in commutative algebra.

In particular, when $Y=0$,
\[ nX\leq \bigvee_{i=0}^{n-1}iX \quad \Rightarrow \quad X\leq 0, \]
and is equivalent to the stipulation that $0$ be integrally closed in $\alpha$. The condition that the additive identity be integrally closed is particularly significant for its interpretations in commutative algebra and monoid theory (cf. aside \ref{ALG_RMK}, \ref{INT_MON}).

\

More generally, one can transform the question of integral closure of an arbitrary element $X\in\alpha$ into the corresponding question for an additive identity.

\begin{lemma}\label{INT_ZERO0}The following are equivalent for any semiring $\alpha$:
\begin{enumerate}\item bounded elements of $\alpha$ are integrally closed;
\item for any bounded localisation $\alpha\rightarrow\alpha[-S]$, $0$ is integrally closed in $\alpha[-S]$.\end{enumerate}\end{lemma}
\begin{proof}The order ideal generated by $nX$ is the preimage of $0$ under the map
\[ \alpha\stackrel{-nX}{\rightarrow}\alpha[-X] \]
of $\alpha$-modules. By the scalar restriction and contraction properties, if $0$ is integrally closed in $\alpha[-X]$ (as a module over itself), then $nX$ is integrally closed in $\alpha$.

If $X$ is bounded, then by stability of integral closure for bounded localisations, the converse is true.\end{proof}

\begin{remark}The above argument is inspired by a situation in commutative algebra, described below in the aside \ref{ALG_RMK}. To emphasise this parallel, it can be rephrased in terms of a \emph{Rees algebra}
\[ \alpha^\circ[X+T]\subseteq \alpha[T], \]
the graded subring of the free $\alpha$-algebra consisting joins of expressions $F_n + nT$ with $F_n\leq nX\in\alpha$. The $n$th graded piece is isomorphic with the slice set $\alpha_{\leq nX}$. In particular, the zeroth term is the semiring of integers $\alpha^\circ$ of $\alpha$. It is an order ideal in $\alpha[T]$.

If we define integral closure of a homogeneous order ideal in $\alpha[T]$ to be integral closure in $\alpha$ of each of the graded pieces, then our argument is to apply a \emph{contraction property} to the quotient $\alpha[X+T]\twoheadrightarrow \alpha[-X]$. This statement is the semiring analogue of the fact that a blow-up of a normal scheme a is normal if and only if all large powers of the ideal being blown up are integrally closed (cf.\ \cite[Prop. 5.2.1]{SwansonHuneke} and aside \ref{ALG_RMK}).
\end{remark}

Under a certain `approximation' hypothesis, detailed in the appendix \ref{append} - satisfied, for example, when every element of $\alpha\setminus\{-\infty\}$ is bounded, or when $\alpha^\circ$ is $T$-adically complete with respect to some $T$ invertible in $\alpha$ - checking integral closure of all bounded elements is enough for normality of the entire semiring.

\begin{lemma}\label{INT_APPROX_NORMAL}Suppose that $\alpha$ satisfies the bounded approximation property (def.\ \ref{APPEND_DEF}). If all bounded elements of $\alpha$ are normal, then $\alpha$ itself is normal.\end{lemma}
\begin{proof}Let $X\in\alpha$, and let $f:\alpha\rightarrow\alpha[-S]$ be a bounded localisation. Since $S$ is bounded, so is $X\vee S$, and so by hypothesis, $X\vee S$ is integrally closed:
\[ X\vee S\leq (f^{-1}fX) \vee S \leq f^{-1}f(X\vee S) =X\vee S \]
i.e.\ $X\vee S=(f^{-1}fX)\vee S$ for all bounded $S$. The approximation property then implies that $f^{-1}fX=X$. Thus $X$ is integrally closed.\end{proof}

\begin{prop}\label{INT_ZERO}Suppose that $\alpha$ has the bounded approximation property. The following are equivalent:
\begin{enumerate}\item $\alpha$ is normal;
\item for any bounded localisation $\alpha\rightarrow\alpha[-S]$, $0$ is integrally closed in $\alpha[-S]$.\end{enumerate}\end{prop}
\begin{proof}By combining lemmas \ref{INT_ZERO0} and \ref{INT_APPROX_NORMAL}.\end{proof}

\section{Commutative algebra}\label{INT_ALG}

Let $A^+$ be a commutative ring, $A$ a localisation of $A^+$. Assume that $A^+\rightarrow A$ is injective. Let $M$ be an $A$-module. A \emph{fractional submodule} of $M$ is a finitely generated $A^+$-submodule.

More generally, if $X^+$ is a (quasi-compact, quasi-separated) scheme, $Z$ a collection of Cartier divisors, $X=X^+\setminus Z$, and $M$ a quasi-coherent $\sh O_X$-module, then these data are locally of the form as above. A \emph{fractional submodule} of $M$ is then a locally finite type $\sh O_{X^+}$-subsheaf. Even if we only care about the affine case, our definition of integral extensions goes via non-affine schemes.

We denote by $\B(M;A^+)$, resp.\ $\B(M;\sh O_{X^+})$ the set of fractional submodules of $M$.\footnote{The prefix for this object was $\B^c$ in \cite{mac}, and $\B$ stood for its lattice completion.} In particular, $\B[A;A^+]$ is an idempotent semiring under the operations of max (ideal sum) and plus (ideal multiplication), and $\B(M;A^+)$ is a $\B[A;A^+]$-module.

If $f:Y^+\rightarrow X^+$ is a morphism of schemes, there is an induced pullback map
\[ f^{-1}:\B(M;\sh O_{X^+}) \rightarrow \B(f^*M;\sh O_{Y^+}) \]
that takes $N\subseteq M$ to the image of $f^*N\rightarrow f^*M$, and a pushforward
\[ f_*:\B(M;\sh O_{Y^+}) \rightarrow \B(f_*M;\sh O_{X^+}), \]
where $M$ is now an $\sh O_{Y^+}$-module, defined by ordinary pushforward of modules.

\begin{defn}\label{INT_ALG_DEF}We say that $N_1\subseteq N_2\subseteq M$ is an \emph{integral extension} of fractional submodules of $M$ (relative to $Z$) if they become equal as submodules of $M$ after pullback to a modification of $X^+$ along $Z$. More precisely, $N_1\subseteq N_2$ is integral if there exists a finite type blow-up $f:\tilde X^+\rightarrow X^+$, along an ideal cosupported set-theoretically in $Z$, such that the inverse images $f^{-1}N_i$ of the $N_i$ in $f^*M$ are equal.\end{defn}

\begin{remark}\ref{ALG_RMK}If $f$ is a modification along $Z$, then $f_*f^*M\cong M$, since both are pushed forward from $M|_X$. It therefore makes sense to compare $N_i$ and $f_*f^{-1}N_i$ as submodules of $M$. By adjunction, $N_1\subseteq N_2$ is an integral extension if and only if there exists a modification $f$ $f_*f^{-1}N_1$ contains $N_2$.\end{remark}

\begin{lemma}\label{ALG_EXT}An extension of fractional submodules of $M$ is integral if and only if their classes in  $\B(M;\sh O_{X^+})$ are an integral extension of elements.\end{lemma}
\begin{proof}
By the blow-up formula of \cite[Prop. 5.22]{mac}, a blow-up $\tilde X^+\rightarrow X^+$ along a finitely generated ideal $T$ induces a bounded localisation
\[ f^{-1}:\B(M;\sh O_{X^+}) \rightarrow \B(M;\sh O_{X^+})[-T]\cong \B(f^*M;\sh O_{\tilde X^+}), \]
and all bounded localisations of $\B[\sh O_X;\sh O_{X^+}]$ arise in this way. In the language of these semirings, the condition that $N_1\subseteq N_2$ be an integral extension is therefore precisely condition \emph{iv)} of proposition \ref{EXT_PROPS} applied to their classes in the fractional submodule semiring.\end{proof}

\begin{defn}The \emph{integral closure} of $N$ in $M$ is the union of all its integral extensions:
\[ \colim_{f:\widetilde{X}^+\rightarrow X^+} f_*f^{-1}N  \quad \hookrightarrow \quad M. \]
Note that this need not be any longer a fractional submodule by our definitions: it might not be finitely generated (even when $X^+$ is Noetherian and $M=\sh O_X$).

An $\sh O_{X^+}$-submodule of $M$ is said to be a \emph{integrally closed fractional submodule} if it is an integral closure of a fractional submodule. Integrally closed fractional submodules are \emph{not} necessarily fractional submodules. We denote by $\B^\nu(M;A^+)$ the set of integrally closed fractional submodules of $M$.\end{defn}

\begin{prop}\label{ALG_COMP}Let $X\subseteq X^+$ be a pair of schemes, $M$ a quasi-coherent $\sh O_X$-module. Then
\[ \B(M;\sh O_{X^+}) \rightarrow \B^\nu(M;\sh O_{X^+}) \]
is a normalisation of $\B[A;A^+]$-modules.\end{prop}
\begin{proof}Indeed, $\sh L\B(M;A^+)$ is the lattice of all $A^+$-submodules of $M$, and by lemma \ref{ALG_EXT}, $\B^\nu(M;\sh O_{X^+})\subseteq\sh L\B(M;A^+)$ is precisely the set of integral closures of elements of $\B(M;A^+)$.\end{proof}

\

It remains to compare our notion of integral closure to the more traditional ones.

\paragraph{Extensions of rings}An algebra extension of $A^+$ inside $A$ is said to be integral, in the usual terminology, if it is a union of finite algebra extensions. In this section we are mainly considering finitely generated modules; therefore, we will compare our notion of integral extension to that of \emph{finite} algebra extension.

\begin{prop}Let $A^+$ be Noetherian, $A^+\subseteq N$ an extension of fractional submodules of $A$. Let $t\in A^+$ be such that $tN\subseteq A^+$. The following are equivalent:
\begin{enumerate}\item $A^+\subseteq N$ is an integral extension with respect to $Z=(t)$;
\item $N$ is contained in a finite $A^+$-algebra inside $A$.\end{enumerate}\end{prop}
\begin{proof}By the finiteness theorem for projective morphisms, if $f:\tilde X^+\rightarrow X^+=\Spec A^+$ is a finite type blow-up, then $f_*f^{-1}\sh O_{X^+}$ is finite over $A^+$. This proves \emph{i)}$\Rightarrow$\emph{ii)} (cf.\ aside \ref{ALG_RMK}).

Conversely, suppose, without loss of genarlity, that $N$ is an $A^+$-algebra. Every element $f\in N$ can be realised as a global function on the blow-up of $\Spec A^+$ along the finite type ideal $tN\trianglelefteq A^+$. Indeed, on the principal chart $(g\neq 0)$ with $g\in N$, it can be expressed as $(fg)/g\in g^{-1}N$. Therefore $N$ is integral over $A^+$ in the sense of definition \ref{INT_ALG_DEF}.\end{proof}

The Noetherian hypothesis in this result can be dropped by using an approximation argument; cf.\ \cite[lemma 3.19]{rig1}.

\paragraph{Extensions of ideals}Let now $I\trianglelefteq\sh O_{X^+}$ be a finite type ideal, and let $f$ be a blow-up with centre $T\trianglelefteq \sh O_{X^+}$. Then $I\subseteq f_*f^{-1}I$ is an integral extension in the sense of definition \ref{INT_ALG_DEF}, and the Rees algebras $R_*$ fit into a commutative diagram
\[\xymatrix{
\Spec R_{f^{-1}I} \ar[r]\ar[d] & \tilde X^+ \ar[dd]^f \\
\Spec R_{f_*f^{-1}I} \ar[d] \\
\Spec R_I \ar[r] & X^+
}\]
whose vertical morphisms are isomorphisms after pullback to $X$.

The morphism $g:\Spec R_{f^{-1}I}\rightarrow\Spec R_I$ is projective, and hence $g_*\sh O\cong R_{f_*f^{-1}I}$ is - at least when $X^+$ is Noetherian - a finite $R_I$-algebra. In other words, \[R_{f_*f^{-1}I}=\bigoplus_{n\in\N}(f_*f^{-1}I)^nt^n\] sits between $R_I$ and its integral closure in $\sh O_X[t]$ - which, by \cite[Prop. 5.2.1]{SwansonHuneke}, is equal to \[ \bigoplus_{n\in\N}\overline{I^n}t^n, \] where $\overline{I^n}$ denotes the integral closure in the sense of \emph{op. cit.}. In particular, $f_*f^{-1}I \subseteq \overline{I}$.

\begin{prop}\label{ALG_PROP}Suppose that $X$ is Noetherian, and let $I\subseteq J$ be an extension of fractional ideals. The following are equivalent:
\begin{enumerate}\item the extension is integral with respect to any $Z$ containing the zero locus of $J$ (in the sense of definition \ref{INT_ALG_DEF});
\item the extension is integral in the (usual) sense of \cite[Def. 1.1.1]{SwansonHuneke};
\item $I$ becomes equal to $J$ after blowing up $J$.\end{enumerate}\end{prop}
\begin{proof}By definition, \emph{iii)} implies \emph{i)}, and we have seen that integral extensions of fractional ideals satisfy equations of integral dependence.

For the remaining implication, suppose that the elements of $J$ satisfy equations of integral dependence over $I$. By the \emph{reduction criterion} \cite[Cor.\ 1.2.5]{SwansonHuneke}, it follows that
\[ J^{n+1}=IJ^n \]
for $n\gg 0$. Thus, $I$ becomes equal to $J$ after blowing up $J$.\end{proof}

\begin{remark}The implication \emph{i)}$\Rightarrow$\emph{iii)} can be phrased purely semiring-theoretically - it is nothing more than a version of the reduction property \eqref{the_implication} of lemma \ref{SEMIRINGS_LEMMA}. 

In fact, it is even valid without the assumption on $Z$, which corresponds to the assumption in \eqref{the_implication} that $X\vee Y$ be \emph{bounded}. However, I don't know a purely semiring-theoretic \emph{proof} without that assumption - indeed, since the proof here goes via the finiteness theorem for projective morphisms, it is difficult to see how it could translate.\end{remark}

\begin{cor}\label{ALG_COR0}Let $X$ be Noetherian and integral with function field $K$. Let $Z$ be the collection of all Cartier divisors on $X$. An extension of ideals on $X$ is integral with respect to $Z$ if and only if it is integral in the sense of
 \cite[Def. 1.1.1]{SwansonHuneke}.\end{cor}
\begin{proof}There are two cases: either $J=0$, in which case there are no integral extensions with either definition (cf.\ \cite[Rmk.\ 1.1.3, (4)]{SwansonHuneke}), or the zero locus of $J$ is contained in some Cartier divisor, in which case the statement is an application of proposition \ref{ALG_PROP}.\end{proof}

\begin{cor}\label{citeme}Let $X$ be Noetherian and integral with function field $K$. Then
\[ \B[K;\sh O_X]\rightarrow\B^\nu[K;\sh O_X] \]
is a normalisation of semirings, where $\B^\nu[K;\sh O_X]$ is the set of integrally closed - in the sense of \cite{SwansonHuneke} - fractional ideal sheaves of $K$.\end{cor}
\begin{proof}By combining corollary \ref{ALG_COR0} with proposition \ref{ALG_COMP}.\end{proof}

\begin{remark}\label{ALG_RMK}In applications to non-Archimedean geometry, it is often useful to assume that the pair $(X,X^+)$ is relatively normal, that is, $\sh O_{X^+}$ is an integrally closed subring of $\sh O_X$.

Each principal open subset $(s\neq 0)$ of the blow-up of $X^+$ along $T$ defines a `completed localisation' $\sh O_X[T/s]$ of $\sh O_X$. Even if $X^+$ is relatively normal, this localisation may not be. 

By replacing
\[ \sh O_{X^+}[T/s] = \colim[ \sh O_{X^+} \stackrel{s}{\rightarrow} T \stackrel{s}{\rightarrow} T^2 \rightarrow\cdots ] \]
with 
\[ \colim[ \sh O_{X^+} \stackrel{s}{\rightarrow}\overline{T} \stackrel{s}{\rightarrow} \overline{T^2} \rightarrow\cdots ] \]
we obtain a corrected localisation which, by \cite[Prop. 5.2.1]{SwansonHuneke}, is the normalisation. Thus restricting attention to normal fractional ideals solves the problem of determining when the integral closure of $\sh O_{X^+}$ is preserved by blowing up. Compare proposition \ref{INT_ZERO}.
\end{remark}

\begin{remark}[Extensions of modules]There are a couple of inequivalent ways to define integral dependence of modules available in the literature \cite{SwansonHuneke,EHU}. I do not know of any obvious relation between either of these and the notion of integral extensions of submodules defined here, but it seems unlikely that they agree in general.\end{remark}

\section{Minkowski semiring}\label{INT_MON}

Integral closure makes sense also for ideals in commutative monoids, or even, for `$\F_1$-algebra' pairs. Under some simplifying hypotheses, the condition takes on a rather combinatorial flavour.

Everything considered below maps to a more traditional commutative algebra setting by replacing a monoid $Q$ with its monoid ring $\Z[Q]$; the study of monoid ideals then corresponds to that of \emph{monomial ideals} in $\Z[Q]$. Cf. \cite[\S1.4]{SwansonHuneke}.

\begin{defn}Let $Q^+$ be a monoid, written additively, and let $Q$ be an (additive) localisation of $Q^+$. Let us call $(Q;Q^+)$ a \emph{monoid pair}.\footnote{The results of this section are equally valid whether or not we consider monoids with absorbing elements, so for simplicity here I will not make this assumption.}

A \emph{fractional ideal} of $Q$ is a finitely generated $Q^+$-invariant subset.\end{defn}

Two properties of this situation immediately distinguish it from the setting of ordinary commutative algebra: first, the `universal valuation'
\[ Q\rightarrow \B[Q;Q^+] \]
into the fractional ideal semiring is \emph{injective}, and second, that in the absence of any other operations
\[ I\vee J = I \cup J \]
for fractional ideals $I,J\subseteq Q$.

\begin{defn}A $Q^+$-submodule $I\subseteq Q$ is said to be \emph{integrally closed} if $\B(I)$ is integrally closed in $\B[Q;Q^+]$. The set of integral closures of finitely generated $Q^+$-submodules is denoted $\B^\nu[Q;Q^+]$.\end{defn}

With this approach to integral closure, the following consequence is immediate:

\begin{prop}Let $(Q;Q^+)$ be a monoid pair. Then
\[ \B[Q;Q^+] \rightarrow \B^\nu[Q;Q^+] \]
is  normalisation of semirings.\end{prop}
\begin{proof}See the proof of proposition \ref{ALG_COMP}.\end{proof}

We can give an explicit description of $\B^\nu[Q]$ in the case that all elements of $Q$ are bounded in $\B[Q]$, that is, when $Q$ is a group. In this case, $\B^\nu[Q]$ is a bounded semiring, and so by proposition \ref{INT_NORMALISATION} normality is equivalent to cancellativity for all finite elements; that is, $\B^\nu[Q]$  is a universal \emph{cancellative quotient} of $\B[Q]$. 

As a preliminary reduction, let us deal with the `cyclotomic' part of $Q$. So suppose, first of all, that $Q$ is torsion. Then the integral closure of $0\in Q^+$ is all of $Q$. It follows that ${}^\nu\B[Q;Q^+]\cong\B$. More generally, $Q$ fits into a sequence
\[ 0\rightarrow Q^\mathrm{tors} \rightarrow Q \rightarrow Q^\mathrm{tf} \rightarrow 0 \]
with $Q^\mathrm{tf}\rightarrow Q\tens\Q$ injective. The induced map 
\[ \B^\nu[Q]\rightarrow \B^\nu[Q^\mathrm{tf}]\] 
is an isomorphism. So we may as well assume $Q$ is torsion-free.

We will treat $\B[Q]$ by embedding it in $\B[Q\tens\Q]$, the addition on which is nothing more than the \emph{Minkowski sum} of subsets of $Q\tens\Q$.

\begin{lemma}\label{MON_CONVEX}An integrally closed ideal of $\B[Q\tens\Q]$ is convex in $Q\tens\Q$.\end{lemma}
\begin{proof}Conversely, suppose $X\in\mathrm{Conv}(I)$, so
\[ nX = \sum_{i=1}^nY_i \in nI\]
for some $Y_i\in I$. If $I$ is integrally closed, $X\in I$ by divisibility (\eqref{divisible}, lemma \ref{SEMIRINGS_LEMMA}).\end{proof}

In particular, if $0$ is integrally closed in $\B[Q\tens\Q]$ , then $Q^+$ is \emph{saturated} in $Q$, that is,
\[ nX\in Q^+ \quad\Rightarrow\quad X\in Q^+ \]
for any $X\in Q$. This condition is cognate to that of relative normality in non-Archimedean rings, or of normality at zero in semirings.

Compare \cite[Prop. 1.4.6]{SwansonHuneke}.

\paragraph{Polytopes} Let $N$ be a rational affine space, $\Delta\subseteq N$ a bounded, rational polytope. Our monoid $Q$ will be the group of affine functions $N\rightarrow\Q$ with integer slopes. It is an extension of a finite rank free $\Z$-module by $\R$.

Our polytope $\Delta$ is cut out by inequalities $F_i\leq\lambda_i$ for some affine functions $F_i\in Q$ and constants $\lambda_i\in\R$. Let us denote by $\B[Q]$ the free semiring on the monoid $Q$. There is a natural surjective evaluation homomorphism
\[ \B[Q]/(F_i\leq\lambda_i)\stackrel{\mathrm{ev}}{\rightarrow}\CPA_\Z(\Delta,\R) \]
into the semiring of convex, piecewise-affine functions on $\Delta$ with integer slopes.

By rewriting the relation $F_i\leq\lambda_i$ as $\lambda_i-F_i\leq 0$, we can impose it at the level of $Q$ by defining the monoid of functions less than zero on $\Delta$:
\[ Q^+= \N (\lambda_i-F_i)+\R_{\leq0}. \]
It is saturated in $Q$ if and only if the relation $F_i\leq\lambda_i$ is \emph{primitive}, which since $\lambda\in\Q$ simply means that $F_i$ is indivisible.

Then $\B[Q]/(F_i\leq\lambda_i)\cong\B[Q;Q^+]$ is a fractional ideal semiring.

\begin{thm}\label{MON_THM}The evaluation map
\[ \mathrm{ev}:\B[Q;Q^+] \rightarrow\mathrm{CPA}_\Z(\Delta,\R)  \]
 is a normalisation of semirings.\end{thm}
\begin{proof}The function semiring $\mathrm{CPA}_\Z(\Delta,\R)$ is cancellative and therefore a quotient of the normalisation.

Conversely, by lemma \ref{MON_CONVEX}, any integrally closed ideal of $\B[Q\tens\Q]$ is convex, and hence is determined uniquely by its polar set
\[ G^\diamond:=\{ (p,q)| Fp-q\leq0 \text{ for all }F\in G \} \subseteq N\times \Q  \]
which is none other than the \emph{upper convex hull} of the graph of $\mathrm{ev}G$.\footnote{This more usual definition of the polar lives in the dual vector space $Q^\vee\tens\Q$. One can pass between that and the setting above via the canonical embedding of $N$ in $Q^\vee\tens\Q$ as an affine hyperplane.}\end{proof}

So away from infinity, the difference between `tropical functions' and real functions is normalisation. This resolves the dichotomy of \cite[e.g.\ 6.5]{mac}.

\appendix
\section{Approximation property}\label{append}

\begin{defn}\label{APPEND_DEF}Let $\alpha$ be an idempotent semiring. Let us say that two order ideals $\iota_1,\iota_2\subseteq\alpha$ are \emph{equal away from infinity} if
\begin{align} \iota_1\vee T=&\iota_2\vee T\text{ for all bounded }T\in\alpha.\end{align}
In other words, $\iota_1=\iota_2$ away from infinity if they support the same bounded functions.

If any two ideals of $\alpha$ that are equal away from infinity are in fact exactly equal, we say that $\alpha$ has the \emph{bounded approximation property}.\end{defn}

\begin{eg}\label{APPEND_EGS}Any semiring in which all elements other than $-\infty$ are bounded has the bounded approximation property. In particular, this applies to the semiring of convex, piecewise-affine functions on any compact polytope.\end{eg}

Let $X,Y\in\alpha$. The equation $X\vee T=Y\vee T$ is equivalent to $X$ and $Y$ being equal in the quotient $\alpha/T=-\infty$. Indeed, $(-)\vee T$ is nothing but the composite of the projection together with its \emph{right adjoint}
\[ \alpha \stackrel{\rho}{\rightarrow} \alpha/(T=-\infty) \stackrel{\rho^\dagger}{\rightarrow}\alpha, \]
the second arrow being injective.

Thus, $X=Y$ away from infinity if and only if they have the same image in the `bounded formal completion', defined as follows: let
\[ \widehat\alpha^\circ:= \lim_{T\in\alpha^\mathrm{bdd}}\alpha^\circ/(T=-\infty) \]
be the limit indexed by the partially ordered set $\alpha^\mathrm{bdd}$ of bounded elements, and write \[\widehat\alpha:=\widehat\alpha^\circ\oplus_{\alpha^\circ}\alpha.\]

In particular, if $\alpha^\circ$ is $T$-adically complete with respect to some invertible $T\in\alpha$, then $\alpha$ has the bounded approximation property.

\bibliographystyle{alpha}
\bibliography{INT}
\end{document}